\documentclass[11pt]{amsart}
\usepackage{geometry}                % See geometry.pdf to learn the layout options. There are lots.
\usepackage{graphicx}
\usepackage{amsfonts,amsthm,amsmath,amssymb,latexsym, amscd, euscript}
\usepackage{graphicx,color}
\usepackage[all]{xy}
\usepackage{epsfig}
\usepackage{labelfig}

\usepackage{amssymb}
\usepackage{epstopdf}
\theoremstyle{plain}
\newtheorem{thm}{Theorem}[section]
\newtheorem{cor}[thm]{Corollary}
\newtheorem{prop}[thm]{Proposition}
\newtheorem{lem}[thm]{Lemma}

\newtheorem{rem}[thm]{Remark}

\title{Geodesically Complete Hyperbolic Structures}
\author{Ara Basmajian}
\thanks{The first author was supported in part by  a PSC-CUNY Grant and a Simons foundation  grant;  the second author was supported by the National Science Foundation grant DMS 1102440, PSC-CUNY grant, and a Simons foundation grant.}

\author{Dragomir \v Sari\' c}
\begin{document}
\maketitle

\begin{abstract}
 In the first part  of this work  we explore the geometry of  infinite type surfaces and  the relationship  between   its convex core and space of ends.  In particular, we give a geometric proof of a Theorem due to   Alvarez  and   Rodriquez that a geodesically complete hyperbolic surface is made up  of its convex core with funnels attached  along the simple closed geodesic components and half-planes attached   along simple open geodesic components.  We next consider  gluing infinitely many  pairs of pants along their  cuffs to  obtain an infinite hyperbolic surface. Such a surface is not always complete; for example, if the cuffs grow fast enough and  the twists are small. 
 We prove that there always exists a choice of twists in the gluings such that the surface is complete regardless of the size of the cuffs. This generalizes the examples of Matsuzaki.

In the second part  we consider complete hyperbolic flute surfaces with rapidly increasing cuff lengths and prove that the corresponding quasiconformal Teichm\"uller space is incomplete in the length spectrum metric. Moreover, we describe the twist coordinates and convergence in terms of the twist coordinates on the closure of the quasiconformal Teichm\"uller space.
\end{abstract}

%%%%%%%%%%%%%%%%%%Introduction

\section{Introduction}

A pair of pants is a  metrically complete hyperbolic surface of type $(0,3)$ with  boundary
components being either closed geodesics, called {\it cuffs},  or punctures, and at least one boundary component being geodesic.  The pair of pants is made  geodesically  complete  by attaching  appropriate size funnels to each geodesic  boundary  component. A natural way of creating more complicated hyperbolic surfaces is to glue pairs of pants along their cuffs via an isometry (where the cuffs that are glued necessarily have the same length). The gluings are not unique and they depend on a real parameter called the {\it twist}.

Recall that a Riemannian manifold is geodesically complete if and only if it is metrically complete. On the other hand, if the manifold is not metrically complete or has boundary  it is natural to ask  if it has a  geodesic completion.  Alvarez and  Rodriguez (see \cite{AR})
 showed  that a hyperbolic surface constructed from gluing  pairs of pants that  form  a pants decomposition  $X^{\prime}$
 has a  unique metric completion  to the convex  core of a geodesically complete hyperbolic surface $X$ by  attaching  funnels and closed half-planes;
conversely, any geodesically complete hyperbolic surface is obtained by attaching funnels and half-planes to the convex core of the surface. We state this as   (Theorem \ref{thm: completing a hyperbolic surface}) and supply our own proof.   We sometimes say that $X$ is the {\it geodesic completion} of   $X^{\prime}$. The existence of  half-planes  glued to the convex core appears in \cite{Bas}. 

To understand the relationship between the convex core boundary and  the space  of ends  we are led to the notion of a {\it visible end}.  An end is {\it visible}   if there is an open sets worth of tangent vectors that when  exponentiated  exit the end. A geodesic ray that begins in $C(X)$ and  exits a visible end  $e$ must intersect a  boundary component of the convex core $C(X)$;  such boundary components are   called  the 
 {\it components of the  visible end $e$}  (see section  \ref{subsection: Geometry of ends}  for the definition).   For a finite type surface, as is well-known,   a  boundary component of the convex core is a  simple closed geodesic  which bounds a    funnel. The funnel determines the  visible end.  More generally for an infinite type surface  a  boundary component of the convex core   may  be   a simple geodesic  isometric to the real line which bounds a half-plane (see Corollary \ref{cor: simple boundary components} and table \ref{Ends: Geometric type vs. topological type}). 
 The half-planes determine visible ends.  That there can be several half-planes determining  the same  visible end 
is illustrated by our  {\it flute surface} examples in section  \ref{section: Visible  ends  with  equivalent boundary components.}.   A {\it flute surface} is a sequence of pairs of pants glued in succession along common length boundaries.  The flute surface is not  necessarily geodesically complete  but always has a natural geodesic completion by Theorem  \ref{thm: completing a hyperbolic surface}. It is said to be a {\it tight flute surface}  if in addition all the pants holes that have not been glued along are in fact cusps.  The flute surface has a unique {\it infinite type  end}-it is the limit of the  isolated ends. See section 
\ref{section: Topology and geometry of ends} for the basics on ends and 
section \ref{sec: geodesically complete} for more on flute surfaces  (also cf. \cite{AR}, \cite{Bas}, \cite{Bas2}, \cite{BK}, \cite{Ha-Sus} and figure \ref{Flute surface}). 

Let $\Omega$ be a hyperbolic domain in the complex plane. In  the paper \cite{AR}  it is shown that a topological half disc  in $\Omega$ bounding  a maximal interval on $\partial \Omega$    can be straightened to be    a maximal  hyperbolic half-plane. 
We study such a situation when $\Omega$ is the unit disc minus a countable set of points that accumulate on the boundary. In particular we prove
 {\it for any $n \in \mathbb{N} \cup \{\infty \}$ there exists a tight flute surface whose unique infinite type end is a visible end with $n$ components}. Though this can be proven using the work of \cite{AR} in section  \ref{section: Visible  ends  with  equivalent boundary components.} we give a geometric proof  using   the convex core and the notion of visible ends.  (see Theorem \ref{cor: equivalent components}),

The  half-planes  in the above discussion arise for  infinite type 
ends that are not metrically  (or equivalently  not geodesically) complete. 
Namely, if we glue infinitely many  pairs of pants we obtain an infinite surface which might not be complete as a metric space. Indeed, when the lengths of the cuffs of the glued pairs of pants are going to infinity the distance between two cuffs is going to zero. If we choose the twists to be zero and if the distances between cuffs add up to a finite number then the obtained surface has an open finite length geodesic which  leaves every compact set and thus the surface  is not metrically complete (see \cite{Bas}).
Thus a  natural question is whether there is a choice of twists such that the surface is complete regardless of how large the cuffs are.  On his way to showing the existence of countable Teichm\"uller  modular  groups  Matsuzaki (see \cite{M}) demonstrated that such phenomena exist by demonstrating   examples. We show that this is the case for all possible topologies on the infinite surfaces which arise by different patterns of gluing pairs of pants.  That is,  by choosing the twists judiciously, there is no need to attach half-planes to  the infinite type ends since they are already geodesically complete (see Theorem 
\ref{thm:large cuffs complete}).

\vskip .2 cm

\noindent {\bf Theorem 1.}
{\it  Let  $X^{\prime}$ be a (not necessarily  complete)  hyperbolic surface with a pants decomposition.  Then there exists a choice of twists along the cuffs
 of the pants  so that  the induced  hyperbolic surface  $X$, after possibly adding funnels, is a geodesically complete hyperbolic surface.}

\vskip .2 cm

Shiga \cite{Shi}, Allessandrini, Liu, Papadopuolos, Su,  and Sun (cf. \cite{ALPS}, \cite{ALPS1}, \cite{ALPSS}) and others (cf. \cite{Kin}, \cite{BK}, \cite{Sar3}) have studied the Teichm\"uller spaces (quasiconformal and length spectrum) of infinite surfaces either when there is an upper bound on the cuff lengths or when they are given by an explicit construction. These surfaces are complete either because of the upper bound on the cuff lengths or by the construction. For arbitrary surfaces built from the pairs of pants with unbounded cuffs a choice of twists might lead to an  incomplete surface. A priori, one might think that being complete could impose conditions on the speed that the cuff lengths go to infinity (which may influence the Teichm\"uller theory). Theorem 2  says that the completeness of surfaces does not impose constraints on the speed of convergence to infinity of the cuff lengths and this opens the possibility for studying Teichm\"uller spaces of infinite surfaces of various topological types and geometrical shapes.

\vskip .2 cm

We proceed to analyze Teichm\"uller spaces of flute surfaces which are obtained by gluing pairs of pants in a chain with cuff lengths rapidly increasing. 
More precisely, we say a strictly increasing sequence is  {\it rapidly increasing}  if the sum of  the first $n$ elements  is going to infinity slower than the 
$(n+1)$-st element.  We choose the twists using Theorem 2  such that the obtained hyperbolic surface $X_0$ is complete.
Our main focus are the twists under the limits of the quasiconformal deformations when we fix the lengths of the cuffs.

The {\it quasiconformal Teichm\"uller space} $T_{qc}(X_0)$ consists of all quasiconformal deformations of $X_0$ modulo postcomposition by conformal maps and homotopy. The quasiconformal Teichm\"uller space $T_{qc}(X_0)$ has a natural metric given by the $1/2$ of the logarithm of the smallest quasiconformal constant in the homotopy class of a quasiconformal map. 
The {\it length spectrum Teichm\"uller space} $T_{ls}(X_0)$ consists of all homeomorphic transformation of $X_0$ such that the ratio of the lengths of the corresponding simple closed geodesics is bounded away from $0$ and $\infty$. The length spectrum distance is $1/2$ the absolute value of the logarithm of the ratio of the lengths of the corresponding simple closed geodesics. It is a consequence of an inequality due to Wolpert   
(see \cite{W}) that  $T_{qc}(X_0)\subset T_{ls}(X_0)$. 

The next theorem considers the closure (of the slice with fixed cuff lengths and varying twists of a geodesic pants decomposition) of the quasiconformal Teichm\"uller space 
$T_{qc}(X_0)$  in the length spectrum metric. We obtain that the twists can be proportional to the lengths of the closed geodesics which tend to infinity, (see Theorem 
\ref{thm:closure-twists}).

\vskip .2 cm

\noindent
{\bf Theorem 2.} 
{\it Let $X_0$ be a geodesically complete  tight flute surface built by gluing pairs of pants with rapidly increasing cuff lengths $\{ l_n\}$.
Then the closure $\overline{T_{qc}(X_0)}$ of the quasiconformal Teichm\"uller space $T_{qc}(X_0)$ contains all surfaces with the Fenchel-Nielsen coordinates $\{ (l_n,t_n)\}_n$, where $-Cl_n\leq t_n\leq Cl_n$, for $C>0$, and the lengths $\{l_n\}$ correspond to a marked surface in $T_{qc}(X_0)$.}
\vskip .2 cm 

\noindent
{\bf Remark.} {\it In the above theorem and the theorems that follow 
  the base point  $X_0$  of $T_{qc}(X_0)$ corresponds to a fixed choice of twist parameters where the  twists $\{t_n\}$  satisfy  $0 \leq t_n < l_n$. }   
 \vskip .2 cm

We are able to describe the convergence in (the slice with fixed cuff lengths of) $T_{qc}(X_0)$ with respect to the length spectrum metric (see Theorem \ref{thm:conv_basepoint}).

\vskip .2 cm

\noindent
{\bf Theorem 3.}
{\it
Let $X_0$ be a  geodesically  complete tight  flute surface with twists $\{ t_n\}$ and rapidly increasing cuff lengths $\{ l_n\}$.  Let $X_k$ be  marked hyperbolic surface with cuff lengths equal to $\{ l_n\}$ and twists $t_{X_k}(\alpha_n)=t_n+O(l_n)$. If $\lim_{k\to\infty}t_{X_k}(\alpha_n)=t_n$ for each $n$, then
$X_k$ converges to $X_0$ in the length spectrum metric.
}

\vskip .2 cm

Using Theorem 3  we prove that the closure of $T_{qc}(X_0)$ is strictly larger than $T_{qc}(X_0)$ (see Theorem \ref{thm:larger}).

\vskip .2 cm

\noindent

{\bf Theorem 4.}
{\it  If $X_0$ is  a geodesically complete tight flute surface with rapidly increasing cuff lengths, then the  length spectrum Teichm\"uller space $T_{ls}(X_0)$ is strictly larger than the quasiconformal Teichm\"uller space $T_{qc}(X_0)$. More precisely, $\overline{T_{qc}(X_0)}-T_{qc}(X_0)$ is non-empty.}

\vskip .2 cm

%%%%%%%%%%%%%%%%Notation %%%%%%%%%%%

{\bf Notation  and contents.} For the convenience of the reader in table \ref{Table:notation} we gather some of  the notation  used in this paper.  The  section  listed is the first place aside from section 1 that  the notation appears.  As a matter of convention we  often use the prime notation such as  $X^{\prime}$ to denote  a not necessarily complete hyperbolic surface. 

In section \ref{section: Topology and geometry of ends}  we discuss the basics of the topology of surfaces including the classification of surfaces using the space of ends. Then we move to the basics of the geometry of ends (\ref{subsection: Geometry of ends}).
In section  \ref{sec: Pants decompositions and  the classification of ends} we discuss pants decompositions  and  the relationship between the boundary  components of the convex core and visible ends.  Section  \ref{section: Visible  ends  with  equivalent boundary components.}  has examples of visible ends having more than one component.
In section  \ref{sec: geodesically complete} we address the question of finding a  geodesically complete structure with rapidly increasing cuffs. In Section \ref{sec:flute_rapid} we define a flute surface whose cuff lengths rapidly increase. In section \ref{sec:Teichmuller} we define the quasiconformal and length spectrum Teichm\"uller spaces. In section \ref{sec:closure} we discuss various facts about the Teichm\"uller space of flute surface with rapidly increasing cuff lengths.

%%%%%%%%%%%%Table: Notation%%%%%%%%%%%%%%%%%%%%%
 \begin{table}[htdp]
\caption{Definitions and notation}
\begin{center}
 \begin{tabular}{| l | c | c  |}
\hline\hline 
  {\bf Definition} & {\bf Section}   &{\bf Notation} \\
 \hline
 hyperbolic plane & 2& $\mathbb{H}$  \\
\hline
 unit disc in complex plane &  2  &$\Delta$\\
\hline
 quasiconformal Teichm\"uller space&  7  &$T_{qc}$ \\
 \hline 
 Teichm\"uller distance &  7 &   $d_T$ \\
\hline
 length spectrum Teichm\"uller space& 7 &  $T_{ls}$ \\
\hline
 length spectrum distance& 7  &$d_{ls}$ \\
\hline
  convex core& 2&$C(X)$ \\
\hline
 boundary of convex core&  2 &$\partial C(X)$ \\
\hline
equivalent boundary components of  convex core&  2&$ b_1 \sim b_2$ \\
\hline

space of ends&  2 &$\mathcal{E}_X$ \\
\hline
space of non-planar ends&  2  &$\mathcal{N}_X$ \\
\hline
 visible ends& 2 &$\mathcal{VE}_X$\\
\hline
 limit set & 2 &$\Lambda (\Gamma)$ \\
\hline
 set of discontinuity & 2  &$\Omega (\Gamma)$ \\
\hline
$X$-length  of  $\alpha$& 7  &$\ell_{X}(\alpha)$ \\
\hline
 flute surface &5  &  \\
\hline
 tight flute surface & 5 &  \\
\hline
 rapidly increasing sequence& 6  &  \\
\hline

\end{tabular} 
\end{center}
\label{Table:notation}
\end{table}

%%%%%%%%%%%%%%%%%%%%%%%%%%%%%%%%

%%%%%%%%%%%%%Topology and geometry of ends

\section{Topology and geometry of ends} \label{section: Topology and geometry of ends}

In this section we  discuss some basics on topology and geometry, introduce the concept of a visible end and  set-up notation.  As  references  for the basics on hyperbolic geometry and discrete groups we refer to the books of  Beardon \cite{Bear} and Buser \cite{Bus}.

%%%%%%%%%%%%%subsection {Topological ends}

\subsection{Topology of  ends}

A surface  is of {\it finite topological type} if it has a finitely generated fundamental group. 
Otherwise we say it is of {\it infinite topological type}. The  proof of the classification of infinite type surfaces  can be found in a paper of 
Ian Richards  (\cite{R}).  We refer the reader to the paper (\cite{BK}) for a discussion  on  ends and  notation.  All  surfaces  in this paper are triangulable and orientable.  Since we are interested in Riemann surfaces all of our surfaces satisfy these two assumptions.  

Fix  $X$  a topological   surface with non-abelian fundamental group, and $\{X_k\}$
 a compact exhaustion of $X$.  Let 
  $\mathcal{C}_1 \supset \mathcal{C}_2 \supset \cdots \supset \mathcal{C}_k \supset \cdots$  be a nested sequence of subsets of $X$ so that,  for each $k$,
 $\mathcal{C}_k$ is a connected component of $X-X_k$. Two such sequences $\{\mathcal{C}_k\}$ and $\{\mathcal{C}_{k}^{\prime}\}$ are equivalent if 
 for each subset  $\mathcal{C}_k$, 
 $\mathcal{C}_{k+n}^{\prime}\subset  \mathcal{C}_k$ for large  $n$, and vice-versa. 
 These equivalence classes form the   {\it space  of  ends} denoted $\mathcal{E}_X$. We usually use a representative sequence to denote the equivalence class of an end. 
 We next describe   a basis for the topology on the space of ends.  Let $U$ be a connected component of   $X-X_k$. Define,
 
 \begin{equation} \label{ }
 U^{\ast}=\{e \in \mathcal{E}_X: e=\{\mathcal{C}_k\}\,  \text{and}\, \mathcal{C}_k \subset U, 
 \text{for large}\, 
 k\}.
\end{equation}
The set of all such $U^{\ast}$ form a basis for the topology of  $\mathcal{E}_X$.   The topology of $\mathcal{E}_X$ does not depend on the choice of compact exhaustion. 
The subspace $\mathcal{N}_X \subset \mathcal{E}_X$  denotes   the  subspace of  non-planar ends; an end   $e=\{\mathcal{C}_k\}$  is {\it non-planar}  if  each 
$\mathcal{C}_k$ has infinite genus.  The non-planar ends form a closed subset of 
$\mathcal{E}_X$.

A homeomorphism $f:X \rightarrow Y$ between surfaces induces a
homeomorphism of pairs,
\begin{equation}
f_{*}:\left(\mathcal {N}_X, \mathcal{E}_X\right)  \rightarrow 
\left(\mathcal{N}_Y,\mathcal {E}_Y\right),
\end{equation}
and hence the pair $\left(\mathcal{N}_X,\mathcal{E}_X\right)$ is a topological 
invariant of $X$ (called the {\it end invariants} of $X$). If $X$
is of finite topological type then $X$ is a closed surface with
$|\mathcal{E}_X| <\infty$ points deleted, and $\mathcal{N}_X=\emptyset$. More particularly, we will say that   $e \in \mathcal{E}_X$ is  a {\it finite topological type end}   if   $e$ is planar and  isolated in   $\mathcal{E}_X$. Otherwise, it is an   {\it infinite topological type end}.  Clearly, a surface  is of infinite topological type if and only if there exists an end of infinite topological type.

\begin{thm}[I. Richards, \cite{R}] \label{thm:classificationofsurfaces}
The orientable surfaces $X$ and $Y$ are topologically equivalent
if and only if $\text{genus}(X)=\text{genus}(Y)$ and $\left(\mathcal{N}_X,\mathcal{E}_X\right)$ is homeomorphic (as pairs) to $\left(\mathcal{N}_Y,\mathcal{E}_Y\right)$.
\end{thm}

Let $e$ be an end of the topological surface $X$.  We say that a  sequence of compact sets 
$\{K_i\}$  {\it exit  the end}  $e$,  if   $\{K_i\}$ converges to $e$ in the space $X\cup \mathcal{E}_X$.    By abuse of language, we also say that a path 
$\gamma : [0, \infty) \rightarrow X$ {\it exits the end} $e$ if $\gamma (t)$
converges to $e$ as $t$ goes to infinity.  In the sequel,  we will often be interested in surfaces with a hyperbolic structure  and hence typically  the  $\{K_i\}$ will either 
be sequences of pairs of pants or simple closed geodesics and the paths  $\gamma$ will be geodesic rays.

%%%%%%%%%%%%Subsection:  Geometry of ends

\subsection{Geometry of  ends} \label{subsection: Geometry of ends} 

We denote the real part, resp. imaginary part, of a complex number $z$ by  $\Re({z})$,
resp. $\Im(z)$.  
A {\it funnel} is   a hyperbolic surface with one geodesic  boundary  component which  is isometric to 
$D/<z \mapsto e^{\ell}z>$, where $D=\{z \in U: \Re({z}) \leq 0\}$ has the induced metric as a  subspace of the  upper half-plane model  $U$ of the hyperbolic plane.   Funnels are  annuli with one  geodesic boundary  component whose   length  $\ell$ determines the funnel.
A (standard)  {\it  cusp} is  a hyperbolic surface with one horocyclic   boundary  component   which is isometric to the quotient $\{z: \Im(z) \geq 1\}/<z \mapsto z+1>$,  where  again
$\{z: \Im(z) \geq 1\}$ has the induced metric as a  subspace of the  upper half-plane model.
  It is well-known that any finite type geodesically complete hyperbolic surface has  ends that are either cusps or funnels. 
 We say  that  $Y \subset X$  is a 
{\it geodesic subsurface}  of the hyperbolic surface $X$ if it is a  subsurface with geodesic boundary.

A Riemannian manifold is {\it geodesically complete} if every geodesic can be extended infinitely far in both directions. Geodesic completeness is equivalent to the induced  Riemannian  (metric) distance being  complete. A  geodesically complete hyperbolic surface   $X$  is the quotient of the hyperbolic plane,
$\mathbb{H}$,  by a torsion-free discrete   non-elementary ({\it Fuchsian})  group  
$\Gamma$ of orientation preserving isometries.  The action of $\Gamma$ on the 
ideal boundary of the hyperbolic plane breaks up into the limit set  $\Lambda (\Gamma)$ and the (possibly empty) set of discontinuity, $\Omega (\Gamma)$.  The set of discontinuity  is made up of a countable union of {\it intervals of discontinuity}. It is well-known that the  stabilizer in $\Gamma$ of an interval of discontinuity  is either  generated by a hyperbolic element or is trivial. Only the first possibility  occurs if  $\Gamma$ is finitely generated.  That such a  stabilizer can be trivial if $\Gamma$ is  infinitely generated  is investigated in the paper   \cite{Bas}.  The {\it convex core} of $X$, $C(X)$, is the quotient of the convex hull of the limit set,  $CH(\Lambda (\Gamma))/ \Gamma$. The convex core is the smallest closed convex  subsurface  (with boundary) which  carries all the homotopy.  In particular, all closed geodesics are contained in 
$C(X)$.  Let $X=\mathbb{H}/\Gamma$ be a geodesically complete hyperbolic surface. $X$ (or  $\Gamma$) is said to be of the  {\it  first kind}  if $\Lambda (\Gamma)=\partial \mathbb{H}$; equivalently,  $C(X)=X$. Otherwise it is of the {\it second kind}.  
We say that a sequence of oriented geodesics $\{L_i\}$ in $\mathbb{H}$ converge to the oriented geodesic 
$L$ if the endpoints of the $\{L_i\}$ converge to the endpoints of $L$. That is, the space of 
oriented geodesics  can be identified with 
$\mathbb{S}^1  \times   \mathbb{S}^1 -\{diagonal\}$. Sometimes we are not interested in orientation of the  geodesics  and so we say   $\L_i\}$ converges to $L$ if up to changing orientations the convergence occurs.  On a hyperbolic surface a sequence of geodesics
$\gamma_i$ is said to converge to the  geodesic  $\gamma$ if the geodesics have lifts to 
the hyperbolic plane so that $\tilde{\gamma_i}$ converges to $\tilde{\gamma}$.

To study the end  geometry of a hyperbolic surface we introduce the notion of a visible  end.  An end $e$ of $X$ is said to be  {\it visible} if there exists an open set $V$ in the unit tangent bundle of $X$ so that for any $v \in V$, the induced geodesic ray $g_v$ exits $e$. 
Otherwise, the end is  said to be  {\it non-visible} or {\it complete}.  We denote the visible ends by $\mathcal{VE}_X$. The next lemma   allows us to describe a visible end in three different ways.

 \begin{lem} \label{lem: visible ends}  Let $X$ be a geodesically complete hyperbolic surface. The following are equivalent.
 \begin{enumerate}
 \item $e$ is a visible end,
 \item there exists of point $x \in X$ and a cone of vectors based at $x$ so that their corresponding geodesics rays exit $e$. 
 \item   there exists a geodesic ray   in $X$  that exits $e$  and leaves $C(X)$ in finite time.
 \end{enumerate}

 \end{lem}
 
 \begin{proof}  The equivalence of items (1) and (2) is  clear. For the equivalence of (2) and (3),  suppose  there is a cone of vectors based at  a point of $X=\mathbb{H}/\Gamma$ for which the corresponding geodesic rays exit $e$. Lifting these geodesic rays  to the universal covering  
 $\mathbb{H}$  and noting that they hit $\partial \mathbb{H}$ in an interval it is clear that this interval must be contained in one of the intervals of discontinuity of $\Gamma$. But then  these geodesic rays leave 
 $C(X)$ in finite time. The converse follows from the fact that if one such leaves $C(X)$ in finite time then there is a cone's  worth that does. 
 \end{proof}
  
  If $X$ has a funnel then the end corresponding to the funnel is a visible  end, and corresponds to exactly one component of the complement of $C(X)$.   
Recalling  that an  end $e$ is of finite (topological)  type if it  is isolated in  $\mathcal{E}_X - \mathcal{N}_X$,
  a  finite type end  is not visible if  it corresponds to a cusp of the surface, and visible  if it corresponds to a funnel.  Thus, as is well-known, we have a nice description of the end geometry of a   hyperbolic surface with a finitely generated fundamental group.

 Let $X$ be a geodesically complete hyperbolic surface. We next  define an equivalence relation on the  boundary components of $C(X)$.  Namely,
 two boundary components $b_1$ and $b_2$ of $\partial  C(X)$ are {\it equivalent}, denoted $b_1 \sim b_2$,  if 
 there exist  two geodesic rays based in   $C(X)$ that go out the same end where  one of them crosses  $b_1$ and the other crosses $b_2$.   This is clearly an equivalence relation and we denote the set of boundary components that are equivalent to the boundary component $b$ by $\{b\}$. 
 
 Now given $e \in   \mathcal{VE}_X$,  let $\gamma$ be a geodesic ray based  in 
 $C(X)$ and  exiting  $e$. Then  $\gamma$ must intersect a boundary component, say 
  $b$, of   $C(X)$. This defines a well-defined mapping,
   $B: \mathcal{VE}_X \rightarrow \partial C(X)/ \sim$ given by $e \mapsto \{b\}$ which is easy to see is a bijection.  
   
     Thus  a  visible end of a geodesically complete hyperbolic surface $X$ corresponds to an equivalence class of connected components of the complement of $C(X)$. The boundary components that correspond (by the bijection)  to the visible  end $e$ we call the 
    {\it components} of $e$. Of course  in the case of  a funnel  there is exactly one  component  in its equivalence class. That there can be more than one component is investigated in section \ref{section: Visible  ends  with  equivalent boundary components.}.

%%%%%%%Pants decompositions and the classification of dnds%%%%%%%%

\section{Pants decompositions and  the classification of ends}
\label{sec: Pants decompositions and  the classification of ends}

In  this section  we discuss a geometric classification of ends. This is a consequence of Theorem \ref{thm: completing a hyperbolic surface} which is due to Alvarez and Rodriguez (see \cite{AR}).  We give a different proof of  this  theorem  which suits our point of view and   leads  to  the  geometric classification  of ends. See lemma 2, section 6 of \cite{Bas} for   the   need to use   closed half-planes in order to geodesically complete the convex core. This half-plane is bounded by a simple open geodesic which is the limit of an infinite sequence of simple closed geodesics. Hence the closed half-plane is also maximal.

A {\it topological  pair  of pants}   is a sphere with three disjoint closed discs removed. We  sometimes include the three  boundary circles  as part of our topological pants.   The context should make it clear. 
A {\it geodesic pair of pants}  is a sphere with three disjoint closed discs removed  endowed with a hyperbolic metric where the boundary curves are  geodesic.    We allow the possibility that the pair of pants has one or two   cusps (a so called {\it tight pair of pants}). A pair of pants has a natural geodesic completion to a complete hyperbolic structure where  each  geodesic boundary component  is  completed by  a  funnel. By abuse of language we sometimes    call the  geodesically complete surface  a pair of pants.  More generally,
any surface  $X^{\prime}$  made up of a  finite number of pairs of pants glued along 
common  cuffs has a unique geodesic completion $X$  by adding funnels.  Furthermore,  
$X^{\prime}=C(X)$. In fact, any geodesically complete hyperbolic surface with finitely generated fundamental group arises in this way.  A  {\it topological pants decomposition}  of a surface is a locally finite decomposition  by pairs of pants where the  pants curve are  homotopically distinct and non-trivial. A topological pants decomposition is a  {\it geodesic  pants decomposition}  if the pairs of pants are geodesic pairs of pants. 
 For ease of language, we will often drop  the adjective {\it geodesic}  before the terms  "pair of pants" and "pants decomposition." The context  should make it clear.

As we saw in section  \ref{section: Topology and geometry of ends}
the boundary   of $C(X)$  in $X$  is the union of  simple  closed geodesics and simple open geodesics. Denoting the simple open ones by $\{L_i\}$, consider the surface with boundary $C(X)- \cup \{L_i\}$.

\begin{prop} \label{prop: straightening pants decompositions}
Let $X$ be a  geodesically complete hyperbolic surface.  Every  topological pants decomposition of   $C(X)- \cup \{L_i\}$ can be straightened to a geodesic pants decomposition. 
\end{prop}

 It is a consequence of Richards 
classification result (\cite{R})  that any infinite type surface admits  a topological exhaustion by finite type surfaces. Hence if  $X$ is a geodesically complete hyperbolic surface this fact coupled with   proposition   \ref{prop: straightening pants decompositions} 
supplies us with a short  proof  of the following corollary.

\begin{cor} \label{cor: geodesic pants decomposition} Let $X$ be a geodesically complete hyperbolic surface. Then 
$C(X)-\{L_i\}$ has a geodesic pants decomposition.
\end{cor}

\begin{rem}  
 In the case that there are no  open geodesics on $\partial C(X)$ 
 corollary \ref{cor: geodesic pants decomposition} is a  result in \cite{ALPSS}.
\end{rem}

The proof of proposition  \ref{prop: straightening pants decompositions} follows. 

\begin{proof} Let $X=\mathbb{H} /\Gamma$.  Since any (topological) pants decomposition of a surface with boundary induces an exhaustion  by finite type   (topological) subsurfaces and vice-versa, it is enough to show that a  topological exhaustion by finite type subsurfaces $\{K_n\}$ of  $C(X)- \cup \{L_i\}$ straightens to an exhaustion by finite area geodesic subsurfaces.  To see this  let $Y_n$, for each $n$,
 be the subsurface $K_n$ with boundary curves straightened to geodesics; $\{Y_n\}$ is a  geodesic subsurface of  $C(X)- \cup \{L_i\} \subset X$.  $Y_n$ is also homeomorphic to 
 $K_n$, for each $n$, which already implies that the $\{Y_n\}$ are locally finite and hence the straightened geodesic  pairs of pants are locally finite.

 We  are left to show that the 
 $\{Y_n\}$ cover  $C(X)- \cup \{L_i\}$. Now, by way of the isomorphism between the fundamental group of $X$ and the group $\Gamma$  there exists 
 a torsion-free Fuchsian subgroup $\Gamma_n$ of $\Gamma$ so that 
 $Y_n=C_n/\Gamma$, where  $C_n$ is the convex hull of the limit set of $\Gamma_n$. 
 Moreover the $\{\Gamma_n\}$ can be chosen so that  $\Gamma_n \leq
\Gamma_{n+1}$, and hence  $C_n \subseteq C_{n+1}$.  Note that since 
the fundamental group of  $C(X)- \cup \{L_i\}$ is isomorphic to $\Gamma$, 
$\Gamma =<\Gamma_n>$ and $\Gamma$ keeps  $\overline{\bigcup_{n} C_n}$ invariant. On the one hand,  $\overline{\bigcup_{n} C_n}$ must be contained in the convex hull of  $\Gamma$, and hence the boundary  at infinity of  $\overline{\bigcup_{n} C_n}$ is contained in the limit set of $\Gamma$.  Since the limit set of $\Gamma$ is the smallest $\Gamma$-invariant non-empty closed subset of $\partial \mathbb{H}$  it must be that   the boundary at infinity of $\overline{\bigcup_{n} C_n}$ is  equal to  the limit set. Hence 
$\overline{\bigcup_{n} C_n}$ is  the convex hull of the limit set, and therefore
$\overline{\bigcup_{n} C_n/\Gamma} =C(X)$.  Thus the geodesic subsurfaces $\{Y_n\}$ exhaust   $C(X)- \cup \{L_i\}$.

\end{proof}

The  ends of a hyperbolic surface constructed from finitely many pairs of pants (that is,  a finite type hyperbolic surface)  are well-known to be geometrically either cusps or funnels.  For a hyperbolic surface constructed from an  infinite number of pairs of pants we have,

\begin{thm}  \label{thm: completing a hyperbolic surface}

Let $X^{\prime}$ be a (not necessarily complete) hyperbolic surface constructed from gluing pairs of pants   that form a pants decomposition of 
$X^{\prime}$. Then $X^{\prime}$  has   a unique metric completion to the convex core of a geodesically complete hyperbolic surface $X$ so that 
$X^{\prime} \subset C(X) \subset  X$.
Moreover,  the geodesic completion  of $X^{\prime}$  
 is attained by adding funnels and  closed
 hyperbolic half-planes.  Conversely, any geodesically  complete hyperbolic surface 
 is the geodesic completion of  a (not necessarily  complete) hyperbolic surface  
 $X^{\prime}$ constructed  from gluing pairs of pants  that form a pants decomposition of  $X^{\prime}$.
\end{thm}

\begin{proof}   The  pants decomposition   induces an exhaustion  by finite type  geodesic subsurfaces $\{Y_n\}$.  Observe that each such $Y_n$ being of finite type  has a completion by adding funnels to  boundary geodesics. Hence there exists a torsion-free Fuchsian group $\Gamma_n$ so that $\mathbb{H}/\Gamma_n$ is a complete hyperbolic surface with convex core $Y_n$.
Denote the convex hull of the limit set of $\Gamma_n$ by $C_n$. Since $\Gamma_n \leq
\Gamma_{n+1}$, we have $C_n \subseteq C_{n+1}$.  Next set $C:=\bigcup_{n} C_n$,
$\Gamma :=\lim_{n \rightarrow \infty} \Gamma_n=<\Gamma_n>$,  and note that 
$\Gamma$ is a torsion-free Fuchsian group (see \cite{Bas} for an  infinite version of the combination theorem). 
Since the $\{C_n\}$ are an increasing nested sequence of domains with geodesic boundary, it must be that as $n \rightarrow \infty$  either the geodesics on the boundary $\partial C_n$ go to infinity  or converge to a geodesic.  In the later case we include the possibility that a component  of $\partial C_n$ is a component of $\partial C_{n+k}$, for all $k \geq 0$.
 Thus 
$C$ is a convex subspace of $\mathbb{H}$ where,

\begin{enumerate}
\item  the boundary of $C$ is made up of complete geodesics; that is, 
$C=\mathbb{H}-\bigcup D_i$, where the $D_i$ are  open half-planes. 

\item the interior of $C$ is the universal cover of $X^{\prime}-\partial X^{\prime}$.

\item a boundary component of $C$ in $\mathbb{H}$ has   stabilizer  in $\Gamma$ that is either  generated by a hyperbolic element or is trivial.  In the later case, the boundary component (geodesic) is the limit of axes of simple hyperbolic elements in $\Gamma$.

\item $\Gamma$ keeps invariant $\overline{C}$, the closure of $C$ in $\mathbb{H}$.

\item $\Gamma$ does not keep invariant any set smaller than $\overline{C}$.

\item $CH(\Lambda (\Gamma))=\overline{C}$, and hence  $\overline{C}/\Gamma$ is the convex core of $\mathbb{H}/\Gamma$, and the metric completion of $X^{\prime}$.

\end{enumerate}

We can conclude that  $\mathbb{H}/\Gamma$
 is the geodesic completion of  $X^{\prime}$ obtained  by attaching closed half-planes and funnels. The  closed half-planes coming  from the half-planes $\{D_i\}$  in item (1). 

For the converse, suppose $X$  is a geodesically complete hyperbolic surface and set 
$X=\mathbb{H}/\Gamma$, where $\Gamma$ is a torsion-free discrete group. As a consequence of the fact that the boundary of $CH(\Gamma)$ in $\mathbb{H}$  is comprised of geodesics that are either axes of  hyperbolic elements   or have  trivial stabilizer, we can conclude that the boundary of $C(X)$ in $X$ is made-up of closed geodesics and (infinite) open geodesics. we denote the infinite open geodesics  by $\{L_i\}$.  Then consider $X^{\prime}=C(X)-\cup \{L_i\}$  and  note that we are retaining  the closed geodesics, if any,  on the boundary of $C(X)$. This  surface with possible boundary, where the boundary components are 
simple closed geodesics, admits  a topological pants decomposition, and  by proposition
\ref{prop: straightening pants decompositions} we straighten this pants decomposition to a geodesic pants decomposition. 
\end{proof}

For $\Gamma$ a non-elementary (that  is, not virtually abelian) Fuchsian group, the stabilizer of an  interval of discontinuity  for $\Gamma$  is non-trivial  if and only if the interval is bounded  by the axis of a  hyperbolic element in $\Gamma$. For trivial stabilizer we have the following characterization,

\begin{cor}  Let $\Gamma$ be a non-elementary  torsion-free Fuchsian group and $I$  an  interval of discontinuity   for $\Gamma$. Then $I$ has  trivial stabilizer in $\Gamma$ 
if and only if    there exists  a sequence of simple hyperbolic elements  in $\Gamma$ whose axes converge to the geodesic  bounding $I$.
\end{cor}

\begin{proof}  Denote by $L$, the geodesic bounding $I$.  If $I$  has trivial stabilizer then Theorem  \ref{thm: completing a hyperbolic surface} implies that there must  exist  a sequence of simple hyperbolic elements whose axes converge  to $L$.  Conversely,  suppose simple hyperbolic elements have axes  $\{L_{n} \}$ converging  to  $L$.  If $L$ were the axis of a hyperbolic element  $\gamma \in \Gamma$ then there would be 
$\gamma$-translates of  $L_n$ that transversely intersect $\{L_n\}$ for large $n$.  This violates the assumption that the $\{L_n\}$ are axes of simple hyperbolic elements. Thus the stabilizer of $L$, and hence $I$, is trivial.  
\end{proof}

Let $X^{\prime}$ be a (not necessarily complete) hyperbolic surface with a pants decomposition  and note that  $X^{\prime} \subset C(X) \subset X$. We have,

\begin{cor}\label{cor: simple boundary components}
{\rm (Visible ends and boundary components)}  Let $X$ be a geodesically complete hyperbolic surface. Then a  boundary component of  $C(X)$ is  either

\begin{enumerate}
\item    a  simple closed geodesic    that   bounds a funnel in $X$  and  corresponds   to an isolated end that is visible or 

\item a  simple geodesic isometric to the real line  that   bounds a half-plane in $X$ and  corresponds to a  component of a visible end of infinite type. 

\end{enumerate}

\end{cor}

 With the aid of this corollary, the following proposition  characterizes the geometry of infinite type ends.

\begin{prop}\label{prop: characterization of geometric ends}
Let $X^{\prime}$ be a (not necessarily complete)  hyperbolic surface  with a pants decomposition and $X$ its geodesic completion. 
\begin{enumerate}

\item  The closure of $X^{\prime}$ in $X$ is $C(X)$, and hence $X$  is of the first kind if and only if $X=X^{\prime}$.

\item  Suppose  $e$ is   an infinite type end of $X$. 
Then $e$ is  not visible    if and only if  for any geodesic ray  $\gamma$
 that exits $e$ the  sequence  of pants in the decomposition of 
$X^{\prime}$ that $\gamma$ intersects   also exit  $e$.

\item $X$ is of the first kind if and only if  each end of $X$ is not visible.

\end{enumerate}

\end{prop}

\begin{proof}   Item (1) follows from Theorem \ref{thm: completing a hyperbolic surface}.

To prove item (2), suppose $e$ is an infinite type end   for $X$ that is not visible. Hence 
by Lemma \ref{lem: visible ends} any  geodesic ray  $\gamma$ that exits $e$ and starts in $C(X)$ must  stay in   $C(X)$. Since  the closure in $X$  of $X^{\prime}$ is $C(X)$ the  geodesic ray  $\gamma$ must pass through infinitely many pair of pants in the decomposition of $X^{\prime}$ that also exit $e$.   On the other hand, if $e$ is a visible end  then, again by  Lemma \ref{lem: visible ends},  there exists a geodesic ray that exits $e$ and leaves the convex core $C(X)$ in finite time. Since this end is of infinite type it must be that this geodesic ray  intersects $\partial C(X)$ in a simple open geodesic.  Then by Theorem  \ref{thm: completing a hyperbolic surface} this ray enters into a half plane embedded in $X$ and hence the infinitely many pairs of pants from 
$X^{\prime}$ that the ray intersects  do not exit the end $e$.  This proves  item (2).

To prove item (3), suppose $X=\mathbb{H}/\Gamma$ is of the second kind and let $I$  be an interval of discontinuity for $\gamma$.  Let $\beta \subset  \mathbb{H}$ be a geodesic ray that limits to a point in $I$. 
 Project to  $X$  the half-plane determined by this interval and the geodesic ray $\gamma$, and note that the projected ray determines an end $e$ for which it exits. 
  If the stabilizer of $I$  is infinite cyclic then the end  $e$ corresponds to  a funnel and hence is visual. If the stabilizer  of $I$ is trivial then  Theorem \ref{thm: completing a hyperbolic surface}   implies again that the end is visible.  To prove the  other direction, suppose $X$ has a visible  end. Then  there exists a geodesic ray that  leaves $C(X)$ in finite time. Hence by Corollary \ref{cor: simple boundary components} this  can only mean that the geodesic ray  intersects the boundary of $C(X)$ and  enters into a half-plane which it can not escape. Thus $X \ne C(X)$ and  we conclude that $X$ is of the second kind.
\end{proof}

We have shown that a visible infinite type end (of a geodesically complete hyperbolic surface)  has an equivalance class of  components  of the convex core boundary being simple open geodesics  with  attached half-planes. 
Table  \ref{Ends: Geometric  type vs. topological type} summarizes  the relationship between   the topology and geometry of an end.

 \vskip10pt

 %%%%%%%ENDS: Table%%%%%%%%%%%%%
 \begin{table}[htdp]
\caption{Ends: Geometry  and Topology.}
\begin{center}

 \begin{tabular}{| c | c | c  | c |}
\hline\hline 
 Top. vs. Geom. end.  & Not visible (complete) & Visible (boundary components)\\
 \hline
   finite type& cusp & simple closed geodesic\\
   \hline 
   infinite type&rays  that exit  end stay in $C(X)$& equiv. class of simple open geodesics\\
\hline
\end{tabular} 
\end{center}
\label{Ends: Geometric type vs. topological type}
\end{table}
 
 %%%%%%%%%%%%%%%%%%%%%End of Table
 
 \vskip10pt

Even though a surface  can have an uncountable number of ends, the hyperbolic metric  places restrictions on the  geometry  of the ends. Namely, the fact that a Fuchsian group has   only a countable number of  intervals of discontinuity implies  that a complete hyperbolic surface  has  at most a countable set of ends  that are visible. In fact,  in 
 section \ref{section: Visible  ends  with  equivalent boundary components.}
we supply examples to show that  a visible end can correspond to 
$n$ equivalent boundary components of $C(X)$ for any $n=1,2,3,...$ including $n$ being the cardinality of the integers.

%%%%%%%%%%%%Examples : Visible  ends  with  equivalent  components

\section{Examples: Visible  ends  with  equivalent  components.} 
\label{section: Visible  ends  with  equivalent boundary components.}

Recall that the {\it components}  of a visible  end are the boundary components of the convex core that  correspond  to the end. That is, these components bound  half-planes that correspond to the end.    In this section we construct  examples  to show that a visible end  can have   countable or any  finite number of components.  These examples first appeared  in the paper \cite{AR}.
Our examples are elementary;  namely for each 
$n \in \mathbb{N} \cup \{\infty \}$ we construct a   flute surface with the unique  infinite  type end being a visible end with $n$ components.  

Let $A$   be a countable set of points in the unit disc $\Delta$ that accumulate  to the set 
$K$ on the boundary $\partial \Delta$; the set $K$ is closed in  $\partial \Delta$. We are interested in the plane domain $X=\Delta - A$ with its  unique complete hyperbolic structure   compatible with the complex structure. $X$ is a tight flute surface and we denote its  unique  infinite type end by $e$. We assume that    
$\partial \Delta -K  \neq \emptyset$ and hence   
$\partial \Delta -K$ is the union of  at most countably many open intervals $\{J_i\}$ on the boundary of the unit disc.    Set  $\widehat{X}=X \cup \{J_i\}$  and suppose $\Gamma$ is the torsion-free Fuchsian group so that $\mathbb{H} /\Gamma$ is conformally equivalent to $X$. We denote the covering map by $f: \mathbb{H} \rightarrow  \mathbb{H} /\Gamma$.

\begin{rem} \label{rem:stabilizer trivial}
Since  the plane domain $X$ does not contain a simple homotopy class which bounds an annulus  (in fact, only punctures and discs), we can conclude that  the stabilizer of any interval of discontinuity   is trivial.  Or equivalently,   $\Gamma$ moves any half-plane bounding an interval of discontinuity  disjointly away from itself.  
\end{rem}

\begin{lem} \label{lem:accumulation}
 If  $\{z_n\}$ is a sequence in $\mathbb{H}$ that accumulates to an interval of discontinuity, then $\{f(z_n)\}$ must accumulate to $\partial \Delta$.
\end{lem}

\begin{proof} We first note that  since the $\{z_n\}$ eventually enter into the complement of the convex hull of $\Gamma$ and $f$ restricted to any  component of the complement is injective (remark  \ref{rem:stabilizer trivial}),  it must be that  $\{f(z_n)\}$ must leave every compact subset of $X$.  Then the only way  $\{f(z_n)\}$ does not accumulate to  $\partial \Delta$ is if there is  a subsequence of $\{f(z_n)\}$  which converge to one of the points of $A$. But since the points of $A$ are punctures, there would have to be  a subsequence of the  
$\{z_n\}$ that limit to a parabolic fixed point.  Since a parabolic fixed point cannot be contained in an interval of discontinuity we have a contradiction.  
\end{proof}

\begin{prop}\label{thm: extension one to one}
 The covering map $f: \mathbb{H}  \rightarrow X$ analytically extends to the set of discontinuity of $\Gamma$ so that the extension  $F$ maps an interval of discontinuity to one of the $\{J_i\}$. Moreover, 
$F:  \mathbb{H} \cup \Omega (\Gamma) \rightarrow   \widehat{X}$ satisfies,

\begin{enumerate}
\item $F$  restricted to any  interval of discontinuity is injective,  

\item  $F$ establishes a one to one correspondence between $\Gamma$-equivalence  classes of  intervals of discontinuity and the intervals  $\{J_i\}$.
  
 \end{enumerate}
  
\end{prop} 

\begin{proof}  Using Lemma \ref{lem:accumulation}, we can apply the reflection principle 
(see \cite{Ah})  to conclude that  $f: \mathbb{H}  \rightarrow X$  has an analytic extension  which maps an interval of discontinuity to one of the intervals $\{J_i\}$.
We denote the extension by $F:  \mathbb{H} \cup \Omega (\Gamma) \rightarrow   \widehat{X}$.  Furthermore, since $f$ restricted to any half-plane $Y$ that bounds an interval of discontinuity $I$ is injective  (remark \ref{rem:stabilizer trivial}) and since $F$ is an analytic extension,  it must be that $F$ is injective  on $I$. This verifies item (1).

So far we have shown that $F$ injectively maps each interval of discontinuity into one of the  $\{J_i\}$. To verify item (2), we need to show that this map is onto one of the $\{J_i\}$, and that each $\{J_i\}$ occurs as the image of an interval of discontinuity. To this end,
suppose  $J$ is  one of the open intervals in  $\{J_i\}$ and let $\beta$ be a  simple differentiable curve in $X$  with endpoints being the endpoints of $J$. Denoting $J$ with its endpoints as $\bar{J}$ we have that  $\bar{J} \cup\beta$ is the boundary of a simply connected region $X_J\subset X$ (cf. Figure \ref{figure:plane domain}). The set 
 $f^{-1}(X_J)$ has countably many simply connected components and the restriction of $f$ to each component is  a one-to-one conformal map by the simple connectedness. Fix one such component $Y$ and define $\phi :X_J\to Y$ by $\phi =(f|_{X_J})^{-1}$.
 Now let $\{z_n\} \in X_J$ such that 
 $\{z_n\}$ accumulate to $J$ as $n\to\infty$. Then,  since 
 $f \circ \phi =\text{id}|_{X_J}$, it must be that
 $\{\phi (z_n)\} \in Y$ accumulate on  $\partial \mathbb{H}$.
The reflection principle applies to $\phi$ and hence  there exists an analytic extension 
$\Phi: X_J \cup J  \to\mathbb{H} \cup \partial \mathbb{H}$ of $\phi$. Since $\Phi$ is an analytic extension,  
$J$ is mapped onto the interior of an arc $I'$ of $\partial \mathbb{H}$.   Furthermore since 
$\phi$ is injective it must be that $\Phi$ is injective on $J$, and hence $I'$ is contained in an interval of discontinuity which we  call $I$.   Since by the identity theorem, $\Phi^{-1}=F$, for all $z \in  Y\cup I'$, we may conclude that $I'=I$ and thus $F$ maps $I$ onto $J$.  
  Since the argument above is natural with respect to the action of $\Gamma$  we have shown that each $\{J_i\}$ arises as the image of a $\Gamma$-equivalence class  of  intervals  of discontinuity.
  
 We may conclude  that there is a one to one correspondence between  $\Gamma$-equivalence classes of intervals of discontinuity  and the intervals  $\{J_i\}$.
 \end{proof}

\begin{rem} The proof of theorem  \ref{thm: extension one to one} shows that  the curves 
$\beta$ in $X$   which bound the intervals  $\{J_i\}$  can be chosen to be the boundary components of the convex core.   See figure \ref{figure:plane domain} for an illustration. 

\end{rem}

  %%%%FIGURE  Plane Domain %%%%%%%%%%
  
  \begin{figure}[h]
%\ShowGrid
\leavevmode \SetLabels
\L(.35*.7) $\gamma$\\
\L(.37*.88) $I$\\
\L(.5*.9) $\varphi$\\
\L(.65*.8) $X_I$\\
\L(.65*.95) $J$\\
\L(.1*.5) $\mathbb{H}=$\\
\L(.85*.5) $=X$\\
\endSetLabels
\begin{center}
\AffixLabels{\centerline{\epsfig{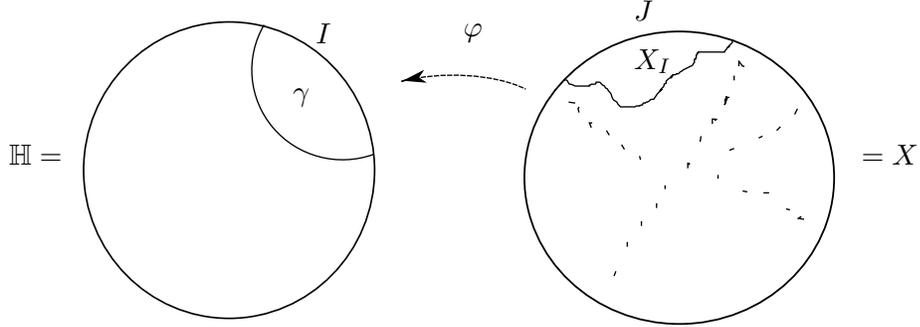} }}
\vspace{-30pt}
\end{center}
\caption{Plane domain.} 
\label{figure:plane domain}
\end{figure}

  %%%%%%%%%%%%%%%%%%%%%%%%%%%

 %%%I don't think we need the argument below. 
% \vskip .2 cm
 %\noindent {\bf Claim.} Let $z_n\in X_J$ such that $z_n$ accumulate to $J$ as $n\to\infty$. Then $w_n=\phi (z_n)\in Y_I$ accumulate on $\partial \mathbb{H}$.
 %\begin{proof}
% Assume on the contrary that $w_n\to w_{\infty} \in \mathbb{H}$ as $n\to\infty$ (after replacing the sequence with a  subsequence). Let $z_{\infty}=f(w_{\infty})$. Let $V_{\infty}$ be a simply connected neighbourhood of $w_{\infty}$ with closure $\bar{V_{\infty}}$ contained in $\mathbb{H}$. Let $U_{\infty}=f(V_{\infty})$ and we have that the closure $\bar{U_{\infty}}$ is contained in $X$ after shrinking $V_{\infty}$ if necessary. Moreover, by definitions we have
 %$$
 %f(V_{\infty}\cap Y)=U_{\infty}\cap X_J.
% $$ 
 %Thus, for $n$ large enough, there exist $z_n'\in U_{\infty}\cap Y$ such that
% $$
% \phi (z_n')=w_n.
 %$$
 %On the other hand, $U_{\infty}$ is compactly contained in $X$ which implies that $z_n\notin U_{\infty}$ for $n$ large enough. Thus $z_n\neq z_n'$ and $\phi (z_n)=w_n=\phi (z_n')$ which is a contradiction with $phi$ being injective. 
% \end{proof}

%\vskip .2 cm

Note that  the plane domain $X$  has exactly one infinite type end and in fact is the geodesic completion of a tight flute surface (see section \ref{sec: geodesically complete} for the definition.  By  
Proposition  \ref{thm: extension one to one} the infinite type  end  is visible and has exactly $n$ components since the components correspond exactly to the intervals $\{J_i\}$. We have established,

\begin{thm} \label{cor: equivalent components} For any $n \in \mathbb{N} \cup \{\infty \}$ there exists a tight flute surface whose unique infinite type end is a visible end with $n$ components. 
\end{thm}

%%%%%%%%%%%%%%%Geodesically Complete Hyperbolic structures

\section{Geodesically complete hyperbolic structures} \label{sec: geodesically complete}

The purpose of this section is to prove the following theorem and investigate some of its
consequences. 

\begin{thm}\label{thm:large cuffs complete}
 Let  $X^{\prime}$ be a (not necessarily complete) hyperbolic surface with a pants decomposition.  Then there exists a choice of twists along the cuffs of the pants 
  so that  the induced  hyperbolic surface  $X$, after possibly adding funnels, is a geodesically complete hyperbolic surface.  
\end{thm}

Thus the induced hyperbolic surface $X$ has the property that all of its infinite ends are not visible.

We will need two lemmas for the proof of Theorem  \ref{thm:large cuffs complete}. 
Let $Y^{\prime}$ be a   finite area  hyperbolic surface with non-empty geodesic boundary, and fix a   boundary component    $\alpha$.   Denote the geodesic completion of 
$Y^{\prime}$ by $Y$. We put an orientation on   $\alpha$  so that the  interior of
 $Y'$ lies to the left  of $\alpha$.
We are interested in unit  vectors based in  $\alpha$ and  directed  to the  interior of 
$Y^{\prime}$. Such a vector $v$  makes an (oriented) angle $\theta$ with $\alpha$, where 
$0 < \theta < \pi$,  and  determines a geodesic ray, 
$g_v: \mathbb{R}_{\geq 0} \rightarrow   Y$, we call  a {\it $\theta$-ray}. A $\pi /2$-ray is also called an {\it orthoray}.
Denote the vector field of such vectors  (based in  $\alpha$)
by $V_{\theta}$. Since the vectors in  $V_{\theta}$ may be identified with points in $\alpha$ (take  base points of  vectors), $V_{\theta}$  inherits a natural measure (called the boundary measure)  and  topology.  This topology is the same  as the  topology of  $V_{\theta}$ as a subspace of the unit tangent bundle of $Y$.
The main lemma for which the construction below hinges  is the following,

\begin{lem}  \label{lem:projection}  Fix   $\theta$,  $0 < \theta < \pi$.  
Then $V_{\theta}= A_{\theta} \dot{\cup}  O_{\theta}$   where 

\begin{enumerate}

\item $A_{\theta}:=\{v \in V_{\theta}: g_v (t)  \subset Y^{\prime}, \text{for all } t \geqslant 0\}$ is a Cantor set of boundary measure zero, 
\item $O_{\theta} :=\{v \in V_{\theta}: g_v (t) \cap \partial Y^{\prime} \neq 
\emptyset,  \text{ for some } t >0 \}$ is a countable union of disjoint open intervals in $\alpha$.

\end{enumerate}

\end{lem}

\begin{proof}  $Y^{\prime}$ is the convex core of the  complete hyperbolic surface, 
$Y=\mathbb{H}/ \Gamma$, where $\Gamma$  is a torsion-free finitely generated Fuchsian group of the second kind. Consider a  connected  oriented  lift  of $\alpha$ to the upper half-plane and  unit tangent vectors  (emanating  to the left) that form an angle $\theta$ with this lift.  The geodesic ray determined by such a vector hits $\partial \mathbb{H}$ in exactly one point.  Conversely any point  on the left side of the lift of $\alpha$ in  $\partial \mathbb{H}$ is the endpoint of such a geodesic ray.  Hence  there exists a projection map $P_{\theta}$  from  
$\partial \mathbb{H}$ to the lift.  When $\theta = \frac{\pi}{2}$ this map is the usual orthogonal projection. The action of $\Gamma$ on  $\partial \mathbb{H}$ breaks up into the (non-empty) set of discontinuity and the limit set.  The set of discontinuity is made of intervals that are bounded by geodesics which are lifts of components of $\partial Y$.
We next  note that the geodesic ray determined by a  $\theta$-vector
$v \in V_{\theta}$ lifts to a geodesic ray that  hits   $\partial \mathbb{H}$ at a point of discontinuity if and only if the vector is in  $O_{\theta}$. Otherwise the geodesic ray hits the limit set of $\Gamma$ and hence  is in $A_{\theta}$. 
Since the limit set of a finitely generated Fuchsian group of the second kind  is known to be  a measure zero Cantor set we have verified items (1) and (2). 
\end{proof}

\begin{lem} \label{lem:complete end}
  Let  $X^{\prime}$ be a (not necessarily complete) hyperbolic surface with a pants decomposition  $\mathcal{P}$,  and let $X$ be  its geodesic completion. Fix $\gamma$  a simple closed geodesic on the hyperbolic surface $X^{\prime}$ and let $e$ be an infinite type end of  $X$.  The  following are equivalent,

\begin{enumerate}
\item $e$ is a visible end,
\item there exists an orthoray based in  $\gamma$  that exits the end $e$ and eventually leaves $C(X)$,

\item there exists an interval $I$  in $\gamma$ for which each  orthoray  based in  $I$  exits the end $e$ and eventually  leaves $C(X)$.
\end{enumerate}
\end{lem}

\begin{proof}  We prove  $1\implies  2\implies 3\implies 1.$
If  $e$ is a visible end  for $X$  then by  definition there exists a geodesic ray that goes out 
$e$ and eventually leaves the convex core $C(X)$.  Hence the geodesic ray must intersect   a component of the boundary of $C(X),$ denoted $\beta$,  and then enter into a half-plane. By Corollary  \ref{cor: simple boundary components},  $\beta$  is necessarily  a simple open geodesic.  We choose a lift of this simple geodesic  
$\tilde{\beta}$ and a lift  $\tilde{\gamma}$ of $\gamma$ and observe   that the orthogeodesic from  $\tilde{\gamma}$ to $\tilde{\beta}$ has finite length.  This orthogeodesic extends to an orthoray that enters the half-plane bounded by 
$\tilde{\beta}$. Pushing this orthoray  to the quotient  $X$  and noting that
this orthoray exits the end $e$ proves  item (2). 
  That there is an intervals worth of such orthorays that hit $\tilde{\beta}$
  follows from the general fact that  the basepoint of an orthogeodesic  between any two geodesics  $\tilde{\gamma}$ to  $\tilde{\beta}$ in $\mathbb{H}$    has an interval of orthorays that all intersect  $\tilde{\beta}$ (See figure \ref{Interval of orthorays}).  Thus  $2\implies 3$. The final implication  $3\implies 1$ is clear.
  \end{proof}

  %%%%FIGURE INTERVAL OF ORTHORAYS %%%%%%%%%%

\begin{figure}[h]
%\ShowGrid
\leavevmode \SetLabels
\L(.37*.5) $\tilde{\gamma}$\\
\L(.62*.5) $\tilde{\beta}$\\
\endSetLabels
\begin{center}
\AffixLabels{\centerline{\epsfig{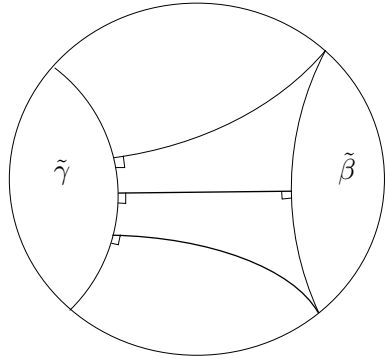} }}
\vspace{-30pt}
\end{center}
\caption{Interval of orthorays.} 
\label{Interval of orthorays}
\end{figure}

  %%%%%%%%%%%%%%%%%%%%%%%%%%%

We are now ready to prove Theorem \ref{thm:large cuffs complete}.
 
\begin{proof}  For each $n$, let     $\{Y_n\}_{n=1}^{\infty}$ be    an exhaustion of  $X^{\prime}$ by finite area geodesic subsurfaces which, to make the argument less cumbersome,   we assume  have the additional property that if a boundary geodesic of $Y_n$ is not a boundary geodesic of $X^{\prime}$ then it is not a boundary geodesic of $Y_{n+1}$. 
Now set  $Z_{n+1}=Y_{n+1}-Y_n$ and note that  in general  $Z_n$ is a finite union of (possibly disconnected)   geodesic subsurfaces.
Fix $\gamma$ a simple closed  oriented geodesic on $Y_1$ and choose an ordered  countable dense subset $A=\{a_n\}_{n=1}^{\infty}$ of 
 $\gamma$.  If $\gamma$ is a boundary geodesic of $Y_1$, by possibly changing the orientation of $\gamma$,  we may assume that $Y_1$ lies to the left of $\gamma$. 
 Consider the orthoray $g_1$ that lies to the left of $\gamma$ and is based at $a_1$. 
  If the ray stays  inside $Y_1$ for all time or if 
it hits $\partial Y_1$ in a component that is a boundary geodesic of $X^{\prime}$ then we do nothing and glue $Z_2$ any way we like. Otherwise,  $g_1$ hits a component, say $\alpha$, of  $\partial Y_1$,  at a point we denote by $p_1$. Denote the  angle  $g_1$ makes with $\alpha$  by $\theta_1$.  Let $Z_2^{*}$ be the component of $Z_2$ which contains  the geodesic that will be identified with $\alpha$. From  
Lemma  \ref{lem:projection}
 the vector field $V_{\pi -\theta_1}$ based in $\alpha$  contains a Cantor set of  vectors,  whose associated geodesic ray stays in $Z_2$. Pick one and call it $v_1$.  Now glue $Z_2$ to $Y_1$ along $\alpha$ so that the orthoray $g_1$ extends smoothly through $p_1$ and into $Z_2$ (see figure \ref{Angles match}).  By construction the ray $g_1$ stays  inside $Y_2=
 Y_1 \cup Z_2$ for infinite time. 
Next  consider the orthoray  $g_2$ with basepoint $a_2$  in the subsurface $Y_2$. As before if  $g_2$ stays inside $Y_2$ or  hits the $\partial Y_2$ in a boundary component 
of $X^{\prime}$ then we do nothing and just glue $Z_3$ anyway we like. Otherwise, 
as before we use Lemma \ref{lem:projection} to glue $Z_3$ to $Y_2$ so that the orthoray 
$g_2$ extends smoothly into $Z_3$ and stays inside $Y_3=Y_2 \cup Z_3$ for infinite time.
We can continue this process ad infinitum so that we have constructed $X$ with 
specified twist parameters  so that the  orthoray which begins at $a_n$ stays in   $Y_{n+1}$.    Now using Lemma \ref{lem:complete end}  we may conclude that all the infinite type ends of $X$ are  not visible.   Finally we add funnels to all the closed boundary geodesics of $X^{\prime}$ and conclude that the resulting surface is geodesically complete. 
\end{proof}

%%%%%%%%%%%%%%%%Figure Angles Match
%Figure : gluing so that angles match up to obtain a smooth geodesic. 

\begin{figure}[h]
%\ShowGrid
\leavevmode \SetLabels
\L(.35*.3) $\theta$\\
\L(.6*.5) $\theta$\\
\endSetLabels
\begin{center}
\AffixLabels{\centerline{\epsfig{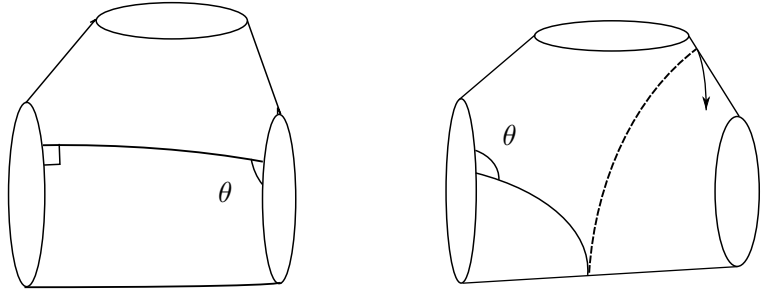} }}
\vspace{-30pt}
\end{center}
\caption{Angles match.} 
\label{Angles match}
\end{figure}

%%%%%%%%%%%%%%%%%%%%%%%%%%%%%%

A {\it flute surface}  is a sequence of pairs of pants glued in succession along common length boundaries.
The flute surface as it stands is not necessarily geodesically complete but always has a natural geodesic completion by Theorem \ref{thm: completing a hyperbolic surface}. 
A flute surface has genus zero, no non-planar ends, and space of ends homeomorphic to 
$\{1/n\}_{n=1}^{\infty} \cup 0$.
Denote the successive cuffs of the flute surface by $\{\alpha_n\}$ as in  figure  \ref{Flute surface}.  We say that a flute surface is {\it tight} if each of the pants holes that have not been glued along are in fact cusps. 

\begin{rem}
A tight flute surface is geodesically complete if and only if its unique infinite type end 
is not visible.
\end{rem}

%%%%%%%%%%%%%%%%%%%%%%%%%%%%%%%
%Figure  :  flute surface with \apha_n curves going to infinite end

\begin{figure}[h]
%\ShowGrid
\leavevmode \SetLabels
\L(.33*.22) ${\alpha}_{1}$\\
\L(.42*.18) ${\alpha}_{2}$\\
\L(.68*.0) ${\alpha}_{n}$\\
\endSetLabels
\begin{center}
\AffixLabels{\centerline{\epsfig{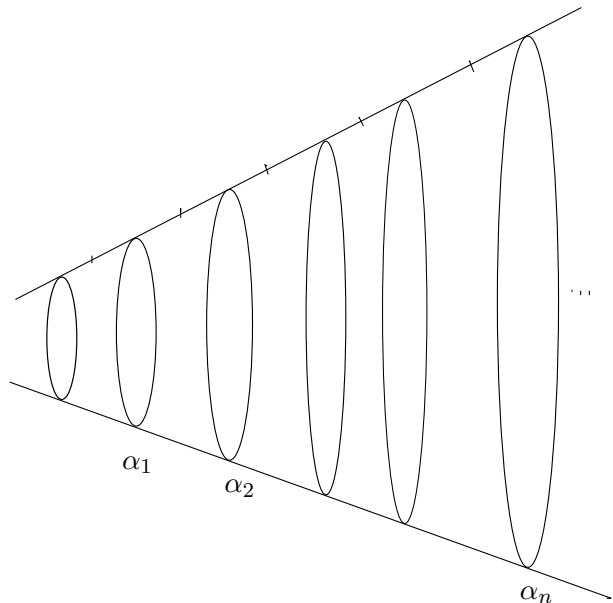} }}
\vspace{-30pt}
\end{center}
\caption{Flute surface.} 
\label{Flute surface}
\end{figure}

%%%%%%%%%%%%%%%%%%%%%%%%%%%%%%%

\begin{thm} Fix any  positive numbers $\{\ell_n\}$, where $\ell_n \rightarrow \infty$. There exists 
 a  tight flute surface of the first kind (that is geodesically complete)   with 
 $\ell(\alpha_n)=\ell_n$, for all $n$.
\end{thm}

\begin{proof}   Recall that a  flute surface has countably many isolated  ends which converge   to an infinite type end. We construct this  surface $X^{\prime}$ out of tight pairs 
of pants,  $\{P_n \}_{n=0}^{\infty}$,  glued in succession.  $P_0$ has cuff lengths,
$(0,0,\ell_1)$ and then generally $P_n$ has cuff lengths  $(\ell_n,0,\ell_{n+1})$ for 
$n=1,2,3,...$. The gluing parameters are chosen as in Theorem \ref{thm:large cuffs complete}.  With these choices  the infinite type end of $X^{\prime}$  is 
of the first kind, and hence all ends of $X$  are of the first kind.  
By Proposition \ref{prop: characterization of geometric ends},
  $X^{\prime}=X$ is a tight flute of the first kind. 
   \end{proof}

  \begin{cor} Let $X$ be the infinite genus surface with one end. 
  There exists a  geodesically complete hyperbolic structure on $X$  for which
  
  \begin{enumerate}
  
  \item the hyperbolic structure is of the first kind. In  particular, the  infinite type end is  not visible (complete),
  
 \item   its  length spectrum is discrete.That is, there are finitely many closed geodesics with length less than any given number.  In particular, any pants decomposition is not upper bounded but is lower  bounded. 
  
  \end{enumerate}
  
  \end{cor}
  
  \begin{proof}   In  the paper  \cite{BK}  it is shown that  for any topological surface there exist hyperbolic structures with a discrete length spectrum. Moreover, the constructions 
  are independent of the twist parameters.  In particular, we can construct  an infinite genus surface with one end using pairs of pants so that it admits a hyperbolic structure with  a discrete length spectrum for any choice of twist parameters. Choosing  the twist parameters  as in Theorem \ref{thm:large cuffs complete} we are guaranteed that the infinite type end is not visible, and hence (since there are no other ends) the hyperbolic structure is of the first kind. 
  \end{proof}

%%%%%%%%%%%%%%%Flute surfaces with rapid increase in cuff lengths

\section{Flute surfaces with rapid increase in cuff lengths}
\label{sec:flute_rapid}

We consider a tight flute surface (see figure  \ref{Flute surface}). That is, we have  a sequence $\{ P_n\}$ of (tight) geodesic pairs of pants whose one cuff is a cusp and two other cuffs are geodesics 
$\alpha_{n}$ and $\alpha_{n+1}$.  Except for the first pants which has two cusps. Denote by $l_n$ the length of $\alpha_n$. We choose one cuff of $P_n$ to be a cusp for the simplicity of the topology of the surface. 

We glue $P_n$ to $P_{n+1}$ by identifying $\alpha_{n+1}$ on $P_n$ with $\alpha_{n+1}$ on $P_{n+1}$. The identification is given by a twist parameter $t_n$ with $0\leq t_n<l_n$. We fix a choice of $t_n$ such that the surface obtained after all the identifications is complete (that is, no visible ends) which is possible by Theorem 2. Thus the obtained surface is  a geodesically complete tight flute surface. There are many choices in the gluings such that the obtained surface is geodesically complete. We fix one such choice and denote the geodesically complete tight  flute surface by $X_0$. 

We  next choose the cuff lengths $\{l_n\}$  of the pants $\{ P_n \}$ to be {\it rapidly increasing}.
That is, $\{ l_n\}$  (strictly)  monotonically goes to $\infty$ and 

\begin{equation}
\label{eq:rapid-growth}
\sum_{i=1}^nl_i=o(l_{n+1}).
\end{equation}

%%%%%%%%%%%%%%%%%Teichmuller spaces

\section{Teichm\"uller spaces}
\label{sec:Teichmuller}

The quasiconformal Teichm\"uller space $T_{qc}(X_0)$ of a geodesically complete hyperbolic surface $X_0$ (without visible ends) consists of all quasiconformal mappings $f:X_0\to X$ modulo post compositions by conformal maps and homotopy. The Teichm\"uller distance between two points $[f:X_0\to X_1]$ and $[g:X_0\to X_2]$ is given by
$$
d_{T}([f],[g])=\frac{1}{2} \inf_{h\sim f\circ g^{-1}}\log K(h)
$$
where the infimum is over all quasiconformal maps $h$ homotopic to $f\circ g^{-1}$ and $K(h)$ is the quasiconformal constant of $h$.

The length spectrum Teichm\"uller space $T_{ls}(X_0)$ of a geodesically complete hyperbolic surface $X_0$ consists of all homeomorphisms $f:X_0\to X$ up to isometry and homotopy, where $X$ is a hyperbolic (not necessarily complete) surface with $\sup_{\alpha}|\log \frac{l_{X}(\alpha)}{l_{X_0}(\alpha )}|<\infty$ and the supremum is over all homotopy classes of simple closed curves $\alpha$.
The length spectrum distance on $T_{ls}(X_0)$ is defined by
$$
d_{ls}([f],[g])=\frac{1}{2}\sup_{\alpha} |\log \frac{l_{f(X_0)}(\alpha)}{l_{g(X_0)}(\alpha )}|
$$
where the supremum is over all simple closed geodesics $\alpha$ on $X_0$. 

The length spectrum Teichm\"uller space $T_{ls}(X_0)$ is complete in the length spectrum metric (cf. \cite{ALPSS}) and $T_{qc}(X_0)\subset T_{ls}(X_0)$. When $X_0$ has a geodesic pants decomposition whose cuff lengths are bounded from above and from below then the length spectrum metric induces the same topology as the Teichm\"uller metric on $T_{qc}(X_0)$ (cf. \cite{Shi}). When $X_0$ has a geodesic pants decomposition with upper bounded cuff lengths and a sequence of cuff lengths goes to zero, then the length spectrum metric is incomplete on $T_{qc}(X_0)$ (cf. \cite{ALPS}) and thus it does not induce the same topology as the Teichm\"uller metric. In the case of upper bounded cuff lengths,  Teichm\"uller spaces $T_{qc}(X_0)$ and $T_{ls}(X_0)$ are parametrized by the Fenchel-Nielsen coordinates (cf. \cite{ALPS1}, \cite{ALPS}), the closure of $T_{qc}(X_0)$ inside $T_{ls}(X_0)$ is described in terms of the Fenchel-Nielsen coordinates (cf. \cite{Sar3}) and local biLipschitz structures of $T_{qc}(X_0)$ and $T_{ls}(X_0)$ is described using the Fenchel-Nielsen coordinates (cf. \cite{ALPS1} and \cite{Sar3}). When $X_0$ has no geodesic pants decomposition with an upper bounded cuff lengths, the parametrization of $T_{qc}(X_0)$ using the Fenchel-Nielsen coordinates is rather challenging. In what follows we describe some aspects of the Fenchel-Nielsen coordinates when $X_0$ is a flute surface with rapidly increasing cuff lengths.

%%%%%%%%%%%%%Section: The closure of $T_{qc}(X_0)%%%%%%%%%%%%

\section{The closure of $T_{qc}(X_0)$}
\label{sec:closure}

In this section we assume that $X_0$ is the  geodesically complete hyperbolic flute surface defined in section 6. Namely, $X_0$ is obtained by gluing tight pairs of pants with cuff lengths $\{ l_n\}$ satisfying
$$
\sum_{i=1}^nl_i=o(l_{n+1}).
$$
In what follows we need the following lemma.

\begin{lem}
\label{lem:pent}
Let $\Sigma$ be a geodesic pentagon with right angles at $A$, $B$, $C$ and $D$, and an ideal vertex at $E$. Let $a$, $b$ and $c$  be the lengths of the sides $AB$, $BC$ and $CD$, respectively. Let $d$ be the length of the geodesic segment orthogonal to both $AB$ and $DE$. Then, for $a>1$ large enough and $c>a$, there exist constants $C_1$ and $C_2$ such that
$$
C_1+c-a\leq d\leq C_2+c-a.
$$
\end{lem}

\begin{proof}
Note that the segment orthogonal to $AB$ and $DE$ is necessarily inside the pentagon $ABCDE$. Denote by $A'$ the endpoint in $AB$ and by $D'$ the endpoint in $DE$ of the segment (cf. Figure 5). From the pentagon $\Sigma$, we get (cf. Beardon \cite[page 159]{Bear}) 
$$
\cosh a\cosh c+1=\sinh a\cosh b\sinh c
$$
which gives
$$
\cosh b=\frac{\cosh a\cosh c+1}{\sinh a\sinh c}.
$$
Further
$$
\frac{\cosh a\cosh c+1}{\sinh a\sinh c}=\frac{(1+e^{-2a})(1+e^{-2c})+4e^{-(a+c)}}{(1-e^{-2a})(1-e^{-2c})}
$$
and Taylor's expansion gives, for some constant $C>0$,
$$
1+Ce^{-2a}\geq\cosh b\geq 1+e^{-2a}.
$$
Consequently, for a constant $C'>0$ which depends on $C$,
$$
C'e^{-a}\geq\sinh b=\sqrt{\cosh^2b-1}\geq\sqrt{2}e^{-a}.
$$
From the right angled pentagon $A'BCDD'$, we get (cf. Beardon \cite[page 159]{Bear}) 
$$
C''e^{-a}e^{c}\geq\cosh d=\sinh b\sinh c\geq\frac{\sqrt{2}}{4}e^{-a}e^{c}
$$
when $c$ is large enough which implies
$$
C'''e^{c-a}\geq e^{d}\geq\cosh d\geq \frac{\sqrt{2}}{4}e^{c-a}.
$$
Taking logarithms in the above inequality, we get
$$
C_2+c-a\geq d\geq \log\frac{\sqrt{2}}{4}+c-a.
$$
\end{proof}

%%%%%%%%%%%%%%Figure PENTAGON

\begin{figure}[h]
%\ShowGrid
\leavevmode \SetLabels
\L(.33*.4) $A$\\
\L(.33*.0) $B$\\
\L(.62*.0) $C$\\
\L(.65*.35) $D$\\
\L(.52*.95) $E$\\
\L(.33*.12) $A'$\\
\L(.6*.48) $D'$\\
\endSetLabels
\begin{center}
\AffixLabels{\centerline{\epsfig{file =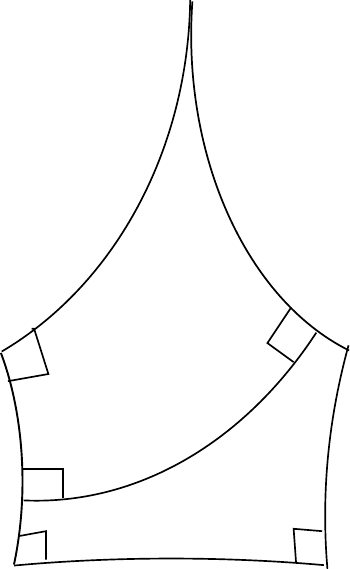,width=4.0cm,angle=0} }}
\vspace{-30pt}
\end{center}
\caption{Lifts.} 
\label{fig1comHypStruc}
\end{figure}

%%%%%%%%%%%%%%

We study the closure $\overline{T_{qc}(X_0)}$ of the quasiconformal Teichm\"uller space $T_{qc}(X_0)$
for the length spectrum metric in the length spectrum Teichm\" uller space $T_{ls}(X_0)$. We establish

\begin{thm}
\label{thm:closure-twists}
 Let $X_0$ be a geodesically complete  tight flute surface built by gluing pairs of pants with rapidly increasing cuff lengths $\{ l_n\}$.
Then the closure $\overline{T_{qc}(X_0)}$ of the quasiconformal Teichm\"uller space $T_{qc}(X_0)$ in the length spectrum metric contains all surfaces with the Fenchel-Nielsen coordinates $\{ (l_n,t_n)\}_n$, where $-Cl_n\leq t_n\leq Cl_n$, for  $C>0$, and the lengths $\{l_n\}$ correspond to a marked surface in $T_{qc}(X_0)$.
\end{thm}

\begin{proof}
Denote by $\{ P_n\}$ the family of tight geodesic pairs of pants that are used to obtain the flute surface $X_0$. Let $\alpha_n$ and $\alpha_{n+1}$ be the cuffs of $P_n$ that are not cusps. Then $\alpha_n$ has length $l_n$.

Let $\{ (l_n,t_n')\}_{n\in\mathbb{N}}$ be the Fenchel-Nielsen coordinates of a marked surface $X$ in $T_{qc}(X_0)$. Define $t_n''=t_n-t_n'$ for some $t_n$ with $-Cl_n\leq t_n\leq Cl_n$ and $C>0$. Let $f_k:X\to X_k$ be a quasiconformal marking map from $X$ to the surface $X_k$ obtained by twists $t_i''$ around $\alpha_i$ on the surface $X$ for $i=1,2,\ldots ,k$. It is clear that $f_k$ can be chosen to be a quasiconformal map since we twist around only finitely many geodesics (cf. \cite{Sar}). We prove that the marked surfaces $X_k$ converge in the length spectrum metric to the surface $X^{*}$ whose Fenchel-Nielsen coordinates are $\{ (l_n,t_n)\}_{n\in\mathbb{N}}$.

We divide each pair of pants $P_n$ into two pentagons with one ideal vertex by three simple geodesic arcs: the first arc, denoted by $\beta_n$, is orthogonal to $\alpha_n$ and $\alpha_{n+1}$ at its endpoints, the second arc, denoted by $\beta_n^1$, is orthogonal to $\alpha_n$ at its endpoint and it finishes in the cusp, and the third arc, denoted by $\beta_n^2$, is orthogonal to $\alpha_{n+1}$ at its endpoint and it also finishes in the cusp. Note that the two pentagons are isometric and that they have four straight angles and one zero angle, i.e. one ideal vertex. Let $b_n$ be the length of $\beta_n$.

Let $\gamma$ be an arbitrary simple closed geodesic on $X$. If $\gamma$ does not intersect any $\{ \alpha_i \}$ for $i\geq k$ then $l_{X_k}(\gamma )=l_{X^{*}}(\gamma )$ and there is nothing to be proved in this case. 

Assume that $\gamma$ intersects $\alpha_i$ for $i=i_0,i_0+1,\ldots ,j_0$ with $k<j_0$. In this case we need to estimate the size of $l_{X^{*}}(\gamma )$ compared $l_{X_k}(\gamma )$. We first estimate the size of $l_{X_k}(\gamma )$.
Note that $\gamma$ intersects the pants $P_{j_0}$ whose boundary geodesics are $\alpha_{j_0}$ and $\alpha_{j_0+1}$. By assumption, $\gamma$ does not intersect $\alpha_{j_0+1}$ which implies that $\gamma$ enters and exists the pants $P_{j_0}$ through $\alpha_{j_0}$. This implies that $\gamma$ necessarily intersects the geodesic arc $\beta_{j_0}^2$ orthogonal to $\alpha_{j_0+1}$ that ends in the puncture because otherwise $\gamma$ could be homotoped such that it does not intersect $\alpha_{j_0}$ which is impossible.

We estimate the length of $\gamma\cap P_{j_0}$. Consider the lift of the situation to the universal covering $\mathbb{H}$. Fix a single component $\tilde{\alpha}_{j_0}$ of the lift of $\alpha_{j_0}$ to $\mathbb{H}$. Denote by $\Sigma_1$ and $\Sigma_2$ the two pentagons that the pants $P_{j_0}$ is divided
into. Consider all lifts of $\Sigma_1$ and $\Sigma_2$ that have one side on $\tilde{\alpha}_{j_0}$. A lift $\tilde{\gamma}$ of $\gamma$ connects $\tilde{\alpha}_{j_0}$ with a lift $\tilde{\beta}_{j_0}^2$ of $\beta_{j_0}^2$ that belongs to a lift of $\Sigma_1$ or $\Sigma_2$ with one boundary side on $\tilde{\alpha}_{j_0}$.   
The length of the segment of $\tilde{\gamma}$ that has one endpoint on $\tilde{\alpha}_{j_0}$ and the other endpoint on $\tilde{\beta}_{j_0}^2$ is greater than the length of the common perpendicular geodesic arc to $\tilde{\alpha}_{j_0}$ and $\tilde{\beta}_{j_0}^2$. It is immediate that the common perpendicular geodesic arc $p$ lies in a single lift of one of the two pentagons (cf. Figure 2). Then Lemma \ref{lem:pent} implies that the length of the common perpendicular arc $p$ is at least $C_1+\frac{l_{j_0+1}-l_{j_0}}{2}$. 
It follows that
\begin{equation}
\label{eq:lowerbd}
l_{X^{*}}(\gamma ),l_{X_k}(\gamma )\geq C_1+l_{j_0+1}-l_{j_0}.
\end{equation}

%%%%%%%%%%%%Figure LIFTED PICTURE
\begin{figure}[h]
%\ShowGrid
\leavevmode \SetLabels
\L(.2*.8) $\tilde{\alpha}_{j_0}$\\
\L(.1*.4) $\tilde{\gamma}$\\
\L(.4*.6) $P$\\
\L(.7*.8) $\tilde{\beta}_{j_0}^1$\\
\L(.55*.2) $\tilde{\beta}_{j_0}$\\
\L(.72*.17) $\tilde{\beta}_{j_0}^2$\\
\endSetLabels
\begin{center}
\AffixLabels{\centerline{\epsfig{file =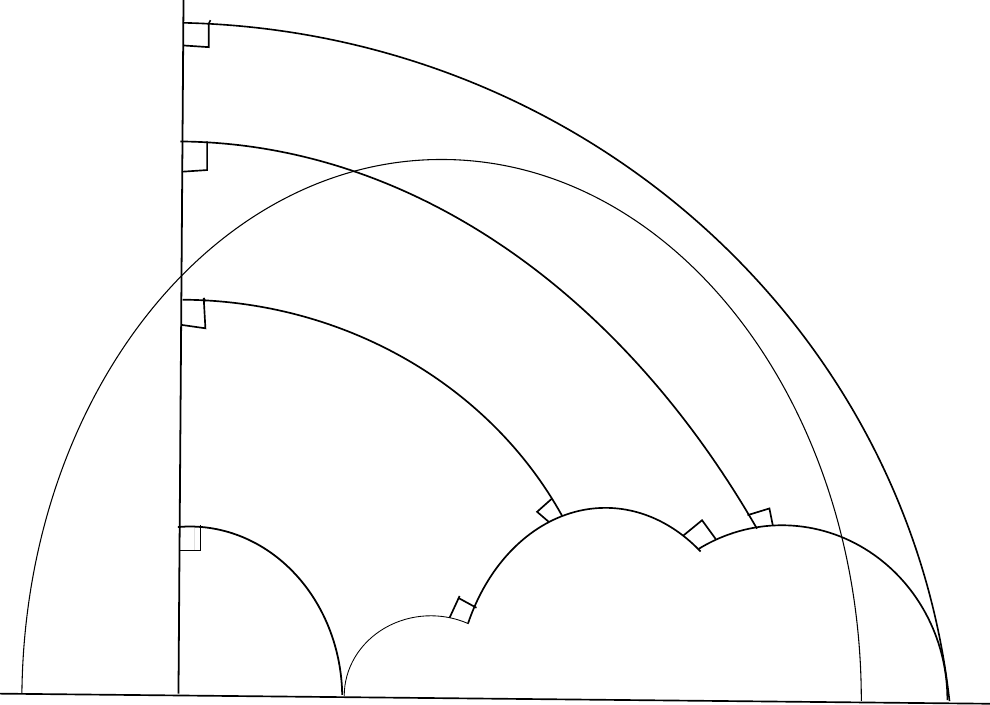,width=12.0cm,angle=0} }}
\vspace{-30pt}
\end{center}
\caption{Lifts.} 
\label{figure2}
\end{figure}

%%%%%%%%%%%%%

Since $X^{*}$ is obtained by twisting around $\alpha_n$ by the amount $t_n''$ and $|t_n''|\leq (C+1)l_n$, we obtain
\begin{equation}
\label{eq:doublebd}
l_{X^{*}}(\gamma )\leq l_{X_k}(\gamma )+\sum_{i=i_0}^{j_0}|t_i''|\leq l_{X_k}(\gamma )+o(l_{j_0+1}).
\end{equation}
We get
$$
\frac{l_{X^{*}}(\gamma )}{l_{X_k}(\gamma )}\leq 1+\frac{1}{l_{X_k}(\gamma )}o(l_{j_0+1})\leq 1+\frac{o(l_{j_0+1})}{C_1+l_{j_0+1}-l_{j_0}}\leq 
1+\frac{o(l_{k+1})}{C_1+l_{k+1}-l_{k}}
\rightarrow 1
$$
as $k\to\infty$ uniformly in $\gamma$.

Notice that $X_k$ is obtained by twisting $-t_i''$ along $\alpha_i$ for $i>k$. If $
\gamma$ is a simple closed geodesic that intersects $\alpha_i$ for $i=i_0,i_0+1,\ldots ,j_0$. If $j_0\leq k$ then
$l_{X_k}(\gamma )/l_{X^{*}}(\gamma )=1$. If $j_0>k$ then
$$
l_{X_k}(\gamma )\leq l_{X^{*}}(\gamma )+\sum_{i=i_0}^{j_0}|t_i''|\leq l_{X^{*}}(\gamma )+o(l_{j_0+1})
$$
which implies as before that
$$
\frac{l_{X_k}(\gamma )}{l_{X^{*}}(\gamma )}\leq 1+\frac{1}{l_{X^{*}}(\gamma )}o(l_{j_0+1})\leq 1+\frac{o(l_{j_0+1})}{C_1+l_{j_0+1}-l_{j_0}}\leq 
1+\frac{o(l_{k+1})}{C_1+l_{k+1}-l_{k}}
\rightarrow 1
$$
as $k\to\infty$ uniformly in $\gamma$. 

We obtained the convergence of $X_k\in T_{qc}(X_0)$ to $X^{*}$ in the length spectrum metric.
\end{proof}

We need the following two lemmas.

\begin{lem}
\label{lem:small_side}
Let $\Sigma$ be a pentagon with four right angles and one ideal vertex, i.e. zero angle. Let $a$, $b$ and $c$ be the lengths of three finite sides of $\Sigma$ in the counterclockwise order. Assume that $c>a>1$, $a\to\infty$ and $c/a\to\infty$.
Then
$$
b=2e^{-a}+o(e^{-a})
$$
where $o(e^{-a})/e^{-a}\to 0$ as $a\to\infty$.
\end{lem}

\begin{proof}
We have (cf. \cite{Bear})
$$
\cosh a\cosh c+1=\sinh a\cosh b\sinh c
$$
which gives
$$
\cosh b=\frac{\cosh a\cosh c+1}{\sinh a\sinh c}=
\frac{(1+e^{-2a})(1+e^{-2c})+4e^{-(a+c)}}{(1-e^{-2a})(1-e^{-2c})}
=1+2e^{-2a}+o(e^{-2a}).
$$
and the result follows.
\end{proof}

\begin{lem}
\label{lem:distance_between_geodesics}
Let $\mathcal{Q}$ be a hyperbolic quadrilateral with three right angles and a fourth angle $0<\phi<\frac{\pi}{2}$. Let $a_1$ and $a_2$ be the lengths of the sides of $\mathcal{Q}$ with two right angles, and let $b_1$ and $b_2$ be the lengths of the opposite sides, respectively.
Then
$$
\sinh b_1=\sinh a_1\cosh b_2
$$
and
$$
\tanh a_1\sinh b_2\tan\phi=1.
$$
\end{lem}

\begin{proof}
From \cite[page 157, Theorem 7.17.1]{Bear}, we have
\begin{equation}
\label{eq:quadrilateral}
\begin{array}l
\sinh a_1\sinh a_2=\cos\phi\\
\cosh a_1=\cosh b_1\sin\phi\\
\cosh a_2=\cosh b_2\sin\phi
\end{array}
\end{equation}
Using (\ref{eq:quadrilateral}) and $\sin^2\phi +\cos^2\phi =1$, we get
$$
\sinh^2a_1\sinh^2a_2+\frac{\cosh^2a_1}{\cosh^2b_1}=1
$$
which implies
$$
\sinh^2a_1\sinh^2a_2\cosh^2b_1+\cosh^2a_1=\cosh^2b_1.
$$
Substituting $\sinh^2a_2=\cosh^2a_2-1$ above we get
$$
\sinh^2a_1\cosh^2a_2\cosh^2b_1-\sinh^2a_1\cosh^2b_1+\cosh^2a_1=\cosh^2b_1.
$$
By diving the third equation with the second equation in (\ref{eq:quadrilateral}), we get $\cosh a_2=\frac{\cosh b_2}{\cosh b_1}\cosh a_1$ and substituting above gives
$$
\sinh^2a_1(\frac{\cosh b_2}{\cosh b_1}\cosh a_1)^2\cosh^2b_1-\sinh^2a_1\cosh^2b_1+\cosh^2a_1=\cosh^2b_1
$$
which in turn gives
$$
\sinh^2a_1\cosh^2 b_2\cosh^2 a_1-\sinh^2a_1\cosh^2b_1-\cosh^2b_1+\cosh^2a_1=0.
$$
Since $\sinh^2a_1+1=\cosh^2a_1$, the above gives
$$
\sinh^2a_1\cosh^2 b_2\cosh^2 a_1-\cosh^2a_1\cosh^2b_1+\cosh^2a_1=0.
$$
Finally, dividing with $\cosh^2a_1$ gives
$$
\sinh^2a_1\cosh^2 b_2-\cosh^2b_1+1=0
$$
and the first formula follows easily.

To prove the second formula, note that (\ref{eq:quadrilateral}) implies
$$
\sinh^2a_2=\frac{\cos^2\phi}{\sinh^2a_1}
$$
and then
$$
\cosh^2 a_2=\sinh^2a_2+1=\frac{\cos^2\phi}{\sinh^2a_1}+1.
$$
By using (\ref{eq:quadrilateral}) above, we get
$$
\frac{\cos^2\phi}{\sinh^2a_1}+1=\cosh^2b_2\sin^2\phi
$$
which gives 
$$
\cos^2\phi +\sinh^2a_1=\sinh^2a_1\cosh^2b_2\sin^2\phi .
$$
Further
$$
\cos^2\phi +\sinh^2a_1\cos^2\phi=\sinh^2a_1\cosh^2b_2\sin^2\phi 
-\sinh^2a_1\sin^2\phi 
$$
and then
$$
\cosh^2a_1\cos^2\phi =\sinh^2a_1\sinh^2b_2\sin^2\phi
$$
and the second formula follows.
\end{proof}

The following theorem proves that pointwise convergence of twists implies the convergence in the length spectrum metric when the lengths of the cuffs in the pants decomposition are fixed.

\begin{thm}
\label{thm:conv_basepoint}
Let $X_0$ be a  geodesically  complete tight  flute surface with twists $\{ t_n\}$ and rapidly increasing cuff lengths $\{ l_n\}$. 
Let $X_k$ be  marked hyperbolic surface with cuff lengths equal to $\{ l_n\}$ and twists $t_{X_k}(\alpha_n)=t_n+O(l_n)$. If $\lim_{k\to\infty}t_{X_k}(\alpha_n)=t_n$ for each $n$, then
$X_k$ converges to $X_0$ in the length spectrum metric.
\end{thm}

\begin{proof}
Let $\epsilon >0$ be given. We need to prove that
$|\frac{l_{X_k}(\beta )}{l_{X_0}(\beta )}-1|<\epsilon$ for all simple closed geodesics $\beta$ on $X_0$ and for all $k\geq k_0(\epsilon )$, where $k_0(\epsilon )>0$ depends on $\epsilon$. 

Indeed, let $\beta$ be a simple closed geodesic on $X_0$. If $\beta$ is a cuff of the pants decomposition, then $l_{X_0}(\beta )=l_{X_k}(\beta )$ and $|\frac{l_{X_k}(\beta )}{l_{X_0}(\beta )}-1|=0$ for all $k$. 

If $\beta$ transversely intersects cuffs, let $\alpha_{n_0}$ be the cuff with the largest index that $\beta$ intersects. 
Then we have
$$
l_{X_0}(\beta )-\sum_{j=1}^{n_0}|t_{X_k}(\alpha_j)-t_j|\leq l_{X_k}(\beta )\leq l_{X_0}(\beta )+\sum_{j=1}^{n_0}|t_{X_k}(\alpha_j)-t_j|
$$
and dividing it with $l_{X_0}(\beta )$ we get
$$
1-\frac{\sum_{j=1}^{n_0}|t_{X_k}(\alpha_j)-t_j|}{l_{X_0}(\beta )}\leq \frac{l_{X_k}(\beta )}{l_{X_0}(\beta )}\leq 1+\frac{\sum_{j=1}^{n_0}|t_{X_k}(\alpha_j)-t_j|}{l_{X_0}(\beta )}.
$$
As in the proof of Theorem \ref{thm:closure-twists}, we have
$$
l_{X_0}(\beta )\geq l_{n_0+1}-l_{n_0}+C.
$$

Note that 
$$
\sum_{j=1}^{n_1}|t_{X_k}(\alpha_j)-t_j|=o(l_{n_1+1})
$$
and we can choose $n_1=n_1(\epsilon )$ such that
$$
1-\epsilon <\frac{\sum_{j=1}^{n}|t_{X_k}(\alpha_j)-t_j|}{l_{n+1}-l_{n}+C}<1+\epsilon
$$
for all $n\geq n_1$.

Since $t_{X_k}(\alpha_n)\to t_n$ as $k\to\infty$ for each $n$, it follows that for any $n_0\leq n_1$ the sum $\sum_{j=1}^{n_0}|t_{X_k}(\alpha_j)-t_j|<\epsilon$ for all $k\geq k_0=k_0(\epsilon )$. Therefore, $\frac{l_{X_k}(\beta )}{l_{X_0}(\beta )}$ is $\epsilon$-close to $1$ for $n_0\leq n_1$ with $k\geq k_0$ large enough. 

If $n_0>n_1$ then 
$$\frac{\sum_{j=1}^{n_0}|t_{X_k}(\alpha_j)-t_j|}{l_{X_0}(\beta )}\leq \frac{o(l_{n_0+1})}{l_{n_0+1}-l_{n_0}+C}\leq \frac{o(l_{n_1+1})}{l_{n_1+1}-l_{n_1}+C}<\epsilon$$
by the choice of $n_1$ for all $k$. It follows that $|\frac{l_{X_k}(\beta )}{l_{X_0}(\beta )}-1|<\epsilon$ when $k>k_0$. The convergence in the length spectrum metric is proved.
\end{proof}

Using Theorem \ref{thm:closure-twists} we obtain

\begin{thm}
\label{thm:larger}
 If $X_0$ is  a geodesically complete tight flute surface with rapidly increasing cuff lengths, then the  length spectrum Teichm\"uller space $T_{ls}(X_0)$ is strictly larger than the quasiconformal Teichm\"uller space $T_{qc}(X_0)$. More precisely, $\overline{T_{qc}(X_0)}-T_{qc}(X_0)$ is non-empty.
\end{thm}

\begin{proof}
Let $\{ t_n\}$ and $\{ l_n\}$ be the twists and the length parameters of $X_0$ for the pants decomposition $\mathcal{P}=\{\alpha_n\}$ as above. Define $t_k(\alpha_n)=t_n+|l_n|$ if $k\leq n$, and define $t_k(\alpha_n)=t_n$ if $k>n$. Define $l_k(\alpha_n)=l_n$ for all $k,n$. The marked surface $X_k$ with the Fenchel-Nielsen coordinates $\{ (t_k(\alpha_n),l_k(\alpha_n))\}$ is a quasiconformal deformation of $X_0$. Indeed, the twists are positive and there are only finitely many of them which implies that the left earthquake given by the positive twists has Thurston bounded earthquake measure. Then the left earthquake induces a quasiconformal deformation (cf. \cite{Sar}). 

Let $X_0'$ be the surface with lengths $\{ l_n\}$ and twists $\{ t_n+|l_n|\}$. Then $X_k$ converges to $X_0'$ as $k\to\infty$ in the length spectrum metric because $t_{X_k}(\alpha_n)=t_{X_0'}(\alpha_n)+O(l_n)$ for all $k,n$ and $t_{X_k}(\alpha_n)=t_{X_0'}(\alpha_n)$ for $n\leq k$ (cf. Theorem \ref{thm:conv_basepoint}). However, the limit $X_0'$ is a not a quasiconformal deformation of $X_0$ since it is obtained by left earthquake with unbounded Thurston's norm (cf. \cite{Sar}).
\end{proof}

\author{Ara Basmajian, Department of Mathematics, Graduate Center and Hunter College, CUNY, abasmajian@gc.cuny.edu}

\author{Dragomir \v Sari\' c, Department of Mathematics,  Graduate Center and Queens College, CUNY, Dragomir.Saric@qc.cuny.edu}

\end{document}